\documentclass[reqno]{amsart}
\usepackage{amssymb}
\usepackage{amssymb}
\usepackage{amscd}
\usepackage{amsfonts}
\usepackage{mathrsfs}
\usepackage{mathtools}
\usepackage{hyperref}
\usepackage{geometry}
\usepackage{color}	
\numberwithin{equation}{section}

\theoremstyle{plain}
\newtheorem{theorem}{Theorem}[section]

\newtheorem{lemma}[theorem]{Lemma}
\newtheorem{corollary}[theorem]{Corollary}

\newtheorem{conjecture}[theorem]{Conjecture}

\theoremstyle{definition}

\newcommand{\Z}{\mathbb{Z}}

\newcommand{\myC}{\mathbb{C}}
\newcommand{\N}{\mathbb{N}}

\newcommand{\mmod}{\ {\rm mod}\ }

\DeclareMathOperator{\sign}{sign}

\begin{document}
	
\title{On sums of coefficients of Borwein type polynomials over arithmetic progressions}

\author{Jiyou Li}
\address{Department of Mathematics, Shanghai Jiao Tong University, Shanghai, P.R. China}
\email{lijiyou@sjtu.edu.cn}

\thanks{This work was supported in part by the National Science Foundation of China (11771280).}
\author{Xiang Yu}
\address{Department of Mathematics, City University of Hong Kong, Kowloon Tong, Hong Kong}
\email{xianyu3-c@my.cityu.edu.hk}
\maketitle

\begin{abstract}
	We obtain asymptotic formulas for sums over arithmetic progressions of coefficients of polynomials of the form
	$$\prod_{j=1}^n\prod_{k=1}^{p-1}(1-q^{pj-k})^s,$$
	where $p$ is an odd prime and $n, s$ are positive integers. Let us denote by $a_i$ the coefficient of $q^i$ in the above polynomial and suppose that  $b$ is an integer. We prove that
		$$\Big|\sum_{i\equiv b\mmod 2pn}a_i-\frac{v(b)p^{sn}}{2pn}\Big|\leq p^{sn/2},$$
		where $v(b)=p-1$ if $b$ divisible by $p$ and $v(b)=-1$ otherwise. This improves a recent result of Goswami and Pantangi \cite{GP}.
\end{abstract}

\section{Introduction}

For an odd prime $p$ and positive integers $n, s$, let  the sequence $(a_i)$ be defined by
\begin{equation}\label{eq:P}
\Big(\frac{(q;q)_{pn}}{(q^p;q^p)_n}\Big)^s=\sum_{i=0}^{sn^2(p-1)p/2} a_i q^i,
\end{equation}
where here and henceforth we use the standard definition of $q$-shifted factorial: 
$$
(a;q)_0=1,\quad (a;q)_n=\prod_{k=1}^{n-1}(1-aq^k),\quad n\geq 1.$$
%Note that here $a_{i}$ should be $a_{p,s,n,i}$. The subscripts $p$, $s$ and $n$ are omitted for notational simplicity.
 In 1990, Peter Borwein discovered some intriguing sign patterns of the coefficients $a_i$ for three different cases $(p, s)=(3, 1), (3, 2), (5, 1)$. They have three  repeating sign patterns of $+--$, $+--$ and $+----$, respectively.
Equivalently, the sign of $a_i$ is determined by $i$ modulo $p$.
These conjectures were formalized by Andrews in 1995 \cite{GA}. 
%The first one  is stated as the following and the others can be found in \cite{GA}.
\begin{conjecture}[First Borwein conjecture]
	For the polynomials $A_n(q)$, $B_n(q)$ and $C_n(q)$ defined by
	$$
	\frac{(q;q)_{3n}}{(q^3;q^3)_n}=A_n(q^3)-qB_n(q^3)-q^2C_n(q^3),$$
	each has non-negative coefficients.
\end{conjecture}

\begin{conjecture}[Second Borwein conjecture]
	For the polynomials $\alpha_n(q)$, $\beta_n(q)$ and $\gamma_n(q)$ defined by
	$$
	\Big(\frac{(q;q)_{3n}}{(q^3;q^3)_n}\Big)^2=\alpha_n(q^3)-q\beta_n(q^3)-q^2\gamma_n(q^3),$$
	each has non-negative coefficients
\end{conjecture}

\begin{conjecture}[Third Borwein conjecture]
	For the polynomials $v_n(q),\phi_n(q),\chi_n(q),\psi_n(q)$ and $\omega_n(q)$ defined by
	$$ \frac{(q;q)_{5n}}{(q^5;q^5)_n}=v_n(q^5)-q\phi_n(q^5)-q^2\chi_n(q^5)-q^3\psi_n(^5)-q^4\omega_n(q^5),$$
	each has non-negative coefficients.
\end{conjecture}

Many attempts have been made to prove these conjectures, including several important generalizations; see for instance 
\cite{GA, Be, BW, BS, B, IKS, RP, W1, W2, AZ}.    However, all these conjectures had been open for many years.
It was not until 2019 that  Wang \cite{CW} first gave an analytic proof of the first Borwein conjecture using the saddle point method and a formula discovered by Andrews [\cite{GA}, Theorem 4.1] for the polynomials $A_n(q), B_n(q)$ and $C_n(q)$. It is not clear whether his method  can be applied to other conjectures.  Even for the first Borwein conjecture, a combinatorial or algebraic proof would be  very interesting.

Instead of evaluating $a_i$ directly, it is natural to consider the sums  of the coefficients of polynomials in the Borwein conjectures over arithmetic progressions. Let $d$ be a positive integer divisible by $p$ and let $b$ be an integer. If we define $$
S_{d,b}:=\sum_{i\equiv b \mmod d}a_i,$$
%=S_{p, s, d, b}
then the positivity (negativity) part of the Borwein conjectures follows from the positivity (negativity, respectively) of $S_{d, b}$ for sufficiently large $d$, for example, $d> sn^2(p-1)p/2$. Please note that here $S_{d,b}$ should be $S_{p, s,n, d,b}$. For notational simplicity, the subscripts $p$, $s$ and $n$ are omitted when there is no confusion.

Using estimates of exponential sums, Zaharescu \cite{AZ} first studied $S_{d, b}$ for a large class of $d$. Let us denote by $[x]$ and $\{x\}$ the usual integer part and fractional part of a real number $x$, respectively. For an integer $b$, define $v(b)=p-1$ if $b$ is divisible by $p$ and $v(b)=-1$ otherwise. Zaharescu proved the following result.
\begin{theorem}[Zaharescu]
Let $s,n$ be positive integers, let $p,q$ be two distinct odd primes with $q\leq n$, and let $b$ be an integer number. Then
	$$
	\Big|S_{pq,b}-\frac{v(b)p^{sn}}{pq}\Big|\leq \frac{(p-1)(q-1)p^{s[n/q]-1}2^{sq(p-1)\{n/q\}}}{q}.
	$$
\end{theorem}
For instance, when $(p,s)=(3,1)$, Zaharescu's bound gives
$$
\Big|S_{3q,b}-\frac{2\cdot 3^n}{3q}\Big|\leq \frac{2(q-1)3^{[n/q]-1}2^{2q \{n/q\}}}{q}
$$
for $b$ divisible by $3$. Note that to ensure that this bound is nontrivial, $q$ must be a prime no greater than $n$.

%Thus a new question naturally arises.
%\begin{problem}
%	For  larger $d$, give a reasonable (what is reasonable?) bound for $S_{d, b}$. \textcolor{red}{Assuming $p\mid d$. it might be an interesting question to determine the largest value of $d$ such that
%	$$S_{d,b}=\frac{(p-1)p^{sn}}{d}(1+o(1))$$
%	for $d$ divisible $p$ and 
%	$$S_{d,b}=-\frac{p^{sn}}{d}(1+o(1))$$
%	for $d$ not divisible by $p$.}
%\end{problem}

For the case $(p,s)=(3,1)$, Li \cite{Li} removed the condition that $q$ is a prime and in fact obtained an estimate with a very small error bound.
He showed that
\begin{theorem}[Li]
	Fix $p=3$, $s=1$.  Let $n,b $ be integers with $n\geq 1$. Then
	$$ \left|S_{3n, b}-\frac{v(b) 3^n}{3n}\right|\leq   2^{n}.$$
\end{theorem}

Goswami and Pantangi \cite{GP} generalized this bound to general cases $(p, s)$ and $d=pn$ following Li's argument and Li--Wan's sieving argument. They proved the following result.
\begin{theorem}[Goswami and Pantangi]
	Let $p$ be an odd prime, and let  $n, b$ be integers with $n\geq 1$.  Then
	$$
	\Big|S_{pn, b}-\frac{v(b) p^{sn}} {pn}\Big|\leq p^{sn/2}.
	$$
\end{theorem}

In this paper, we improve Goswami and Pantangi's result to arithmetic progressions with a larger common difference of $2pn$. Our main result is the following:
\begin{theorem}\label{thm:main}
     Let $p$ be an odd prime, and let $n, b$ be integers with $n\geq 1$. Then
	$$\Big|S_{2pn, b}-\frac{v(b)p^{sn}}{2pn}\Big|\leq p^{sn/2}.$$
\end{theorem}

\begin{corollary}
  For an odd prime $p$ and an integer $b$, define
  $$ N_{p,s,b}:=\begin{cases}
  \inf \{n\in \N: (p-1)p^{sn/2-1}>2n\},& \text{if}\ p\mid b;\\
  \inf\{n\in\N: p^{sn/2-1}>2n\},& \text{otherwise}.
  \end{cases}$$
  Then for all $n\geq N_{p,s,b}$, we have 
  $$
  S_{2pn, b}\begin{cases}
  >0, & \text{if}\ p\mid b;\\
  <0, & \text{otherwise}.
  \end{cases}
  $$
\end{corollary}

The main idea of this paper is to shift the problem to a subset-sum type problem (see Section \ref{sec:reduction}), which can be handled by the Li--Wan sieve. We remark that both Goswami and Pantangi's argument and ours rely on the algebraic structures of certain sets in the additive groups $\Z/pn\Z$  and $\Z/2pn\Z$, respectively. When $d$ becomes larger than $2pn$, the argument does not work due to the lack of the algebraic structure. To give a reasonable bound of $S_{b,d}$ for large $d$, one has to obtain non-trivial estimates for exponential sums over a small subset of the additive group $\Z/d\Z$, which in general is a difficult problem in number theory. %However, it is still an interesting question of giving a reasonable estimate for $S_{d,b}$ for larger $d$ for the reason mentioned before. 

%\textcolor{red}{\begin{remark}[This remark will be deleted]
%Assuming $p\mid d$, based on the known results ($d\leq 2pn$),  one may conjecture that 
%$$S_{d,b}\sim \frac{(p-1)p^{sn}}{d} $$
% if $b$ is divisible by $p$, and 
% $$S_{d,b}\sim -\frac{p^{sn}}{d} $$
%if $b$ is not divisible by $p$, where $f\sim g$ means $\lim_{n\to\infty}\frac{f(n)}{g(n)}=1$. For the case $(p,s)=(3,1)$, numerics suggest that this seems to be true if $d=Cpn$, where  $C$ is a constant and false if $d=pn^{\alpha}$ with $\alpha>1$
%\end{remark}}

This paper is organized as follows. In Section 2, we reduce the problem to a subset-sum type problem. The Li--Wave sieve is introduced in Section 3 and some useful lemmas are presented in Section 4. Finally, in Section 5, we prove our main result.

{\bf Notation.} The congruence notion $a\equiv b\ {\rm mod}\ n$ means $a-b$ is divisible by $n$. We use  $|E|$ to denote the cardinality of the set $E$. If $F(x)=\sum_{n=0}^\infty a_nx^n$ is a formal power series, then we use $[x^n]F(x)=a_n$ to denote the coefficient of $x^n$ in $F(x)$. If $S$ is a statement, we use $1_S$ to denote the indicator function of $S$, thus $1_S = 1$ when $S$ is true and $1_S$ = 0 when $S$ is false.

\section{Reduction to a subset-sum type  problem}\label{sec:reduction}

As in \cite{Li} and \cite{GP}, we first reduce the problem to a subset-sum type  problem over the additive group of integers modulo $2pn$. The starting point is the equality $(1-q^j)=-q^j(1-q^{-j})$, allowing us to write  \eqref{eq:P} as
\begin{align}\label{eq:PP}
\Big(\frac{(q;q)_{pn}}{(q^p;q^p)_n}\Big)^s=\quad(-1)^{sn(p-1)/2}q^{e}\prod_{j=-(n-1)}^n\prod_{k=1}^{(p-1)/2}(1-q^{pj-k})^s.
\end{align}
where $e=e_{p,s,n}=sn(p-1)(2pn+1-p)/8$. If we denote by $b_{i}$ the coefficient $q^i$ in the Laurent polynomial
$$
b_i=[q^i]\prod_{j=-(n-1)}^n\prod_{k=1}^{(p-1)/2}(1-q^{pj-k})^s,$$
then we deduce from \eqref{eq:PP} that $a_{i}=(-1)^{sn(p-1)/2} b_{i-e}$. In particular, we have
\begin{align}\label{eq:ab}
S_{2pn,b}=\sum_{i\equiv b\ {\rm mod}\ {2pn}} a_{i}= (-1)^{sn(p-1)/2} \sum_{i\equiv b-e\ {\rm mod}\ {2pn}} b_{i}.
\end{align}
%Thus to estimate $S_{2pn,b}$, it suffices to consider the sum $\sum_{i\equiv b\ {\rm mod}\ {2pn}} b_{i}$.

On the other hand, let $D$ denote the set 
\begin{align}\label{eq:D}
D=\{pj-k:-(n-1)\leq j\leq n,\ 1\leq k\leq \frac{p-1}{2}\}.
\end{align}
Given integers $0\leq m_i\leq |D|$, $1\leq i\leq s$ and $b$,  define the quantity $N_D(m_1,m_2,\dots,m_s,b)$ as
$$
N_{D}(m_1,m_2,\dots,m_s,b):=|\{(V_1,V_2,\dots,V_s): V_i\subset D,\ |V_i|=k_i,\ 1\leq i\leq s,\ \sum_{i=1}^s\sum_{x\in V_i}x\equiv b\ {\rm mod}\ 2pn\}|.
$$
That is, $N_{D}(m_1,m_2,\dots,m_s,b)$ is the number of ordered $s$-tuples of subsets of $D$ with cardinalities $m_i$, $1\leq i\leq s$ whose elements sum to $b$. The well-known subset-sum problem is to count the number of subsets with prescribed cardinality whose elements sum to a given element. Thus the problem of counting $N_{D}(m_1,m_2,\dots,m_s,b)$ can be viewed as a generation of the subset-sum problem and it is the usual subset-sum problem when $s=1$.
  To study the parity property of $N_D(m_1,m_2,\dots,m_s,b)$,  we define $N_D(b)$ to be their alternating sum
\begin{align}\label{eq:ND}
N_D(b):=\sum_{0\leq m_i,\dots,m_s\leq |D|}(-1)^{\sum_{i=1}^s m_i}N_{D}(m_1,m_2,\dots,m_s,b).
\end{align}

Clearly, we have $N_D(b)=\sum_{i\equiv b\ {\rm mod}\ {2pn}} b_i$, which together \eqref{eq:ab} leads to
\begin{align}\label{eq:Sb}
S_{2pn,b}= (-1)^{sn(p-1)/2}N_D(b-e).
\end{align}
The above equality establishes a connection between $S_{2pn,b}$ and $N_D(b)$, and the main problem is then reduced to counting $N_{D}(b)$ or more precisely to counting $N_D(m_1,m_2,\dots,m_s,b)$, which is  a subset-sum type problem over the additive group of integers modulo $2pn$ that can be computed by the sieve formula of Li and Wan. %and Wan and character sums.

\section{The Li--Wan sieve}

For the purpose of the proof, we briefly introduce  the Li--Wan sieve \cite{Li2}. Suppose that $m$ is a positive integer and  $A$ is a finite set. Let $A^m$ be the Cartesian product of $m$ copies $A$. For subset $X$  of $A^m$, let $\overline{X}$ be the set of elements in $X$ with distinct coordinates
$$
\overline{X}=\{(x_1,x_2,\dots,x_m)\in X:x_i\neq x_j,\ \forall\, i\neq j\}.
$$
A sieve formula discovered by Li and Wan \cite{Li2} gives an approach to computing $|\overline{X}|$  through a simpler way than the usual inclusion-exclusion sieve.  The new sieve formula
shows that there exists a large number of cancellations in the summation in the inclusion-exclusion principle. The number of terms in the summation is significantly reduced from $2^{{m \choose 2}}$ to $m!$.

 Let $S_m$ be the symmetric group on the set $\{1,2,\dots,m\}$. Given a permutation $\tau\in S_m$, we can write it a product disjoint cycles $\tau=C_1C_2\cdots C_{c(\tau)}$, where $c(\tau)$ denotes the number of disjoint cycles of $\tau$. We define the signature of $\tau$ to be $\sign(\tau)=(-1)^{m-c(\tau)}$. We also define the set $X_\tau$ to be
$$
X_\tau=\{(x_1,x_2,\dots,x_m)\in X:x_i\ \text{are equal for}\ i\in C_j, 1\leq j\leq c(\tau)\}.
$$
In other words, $X_\tau$ is the set of elements in $X$ that are invariant under the action of $\tau$ defined by $\tau \circ (x_1,x_2,\dots,x_m):=(x_{\tau(1)},x_{\tau(2)},\dots,x_{\tau(m)})$.  The Li--Wan sieve gives a formula for calculating sums over $\overline{X}$ via sums over $\overline{X}_\tau$.
\begin{theorem}[\cite{Li2}, Theorem 2.6]\label{thm:Li-Wan}
	Let $f:X\to \mathbb{C}$ be a complex-valued function defined over $X$. Then we have
	\begin{align*}
	\sum_{x\in\overline{X}} f(x)=\sum_{\tau\in S_m}\sign(\tau)\sum_{x\in X_\tau}f(x).
	\end{align*}
\end{theorem}

A permutation $\tau\in S_m$ is said to be of type  $(c_1,c_2,\dots,c_m)$ if it has exactly $c_i$ cycles of length $i$, $1\leq i\leq m$. Let $N(c_1,c_2,\dots,c_m)$ denote the number of permutations of type $(c_1,\dots,c_m)$. By the orbit-stabilizer theorem, we have the well-known formula
\begin{align}\label{eq:N}
N(c_1,c_2,\dots,c_m)=\frac{m!}{1^{c_1}c_1!2^{c_2}c_2!\cdots m^{c_m}c_m!}.
\end{align}
If we define an $m$-variate polynomial $Z_m$ via
$$
Z_m(t_1,t_2,\dots,t_m)=\frac{1}{m!}\sum_{\sum ic_i=m}N(c_1,c_2,\dots,c_m)t_1^{c_1}t_2^{c_2}\cdots t_m^{c_m},$$
then it follows from \eqref{eq:N} that $Z_m$ satisfies the generating function
\begin{align}\label{eq:exp}
\sum_{m\geq 0}Z_m(t_1,t_2,\dots,t_m)u^m=\exp\Big(t_1u+t_2\frac{u^2}{2}+t_3\frac{u^3}{3}+\cdots\Big).
\end{align}

\section{Some useful lemmas}

In this section, we present some lemmas that will be used later.
\begin{lemma}[\cite{Li}, Lemma 2.3]\label{lem:a}
Suppose that $\ell$ is a positive integer with	$1\leq \ell\leq m$. If $t_i=a$ if $\ell\mid i$ and $t_i=0$ otherwise, then we have
\begin{align*}
Z_m(t_1,t_2,\dots,t_m)=Z_m(\underbrace{0,\dots,0}_{\ell-1},a,\underbrace{0,\dots,0}_{\ell-1},a,\dots)=[u^m]{(1-u^\ell)^{-a/\ell}}.
\end{align*}

\end{lemma}

\begin{lemma}\label{lem:b}
	Suppose that $\ell $ is a positive integer with	$1\leq \ell \leq m$, and that  $B$ be a finite set of complex numbers. If $t_i=\sum_{z\in B}z^{i}a$ for $\ell \mid i$ and $t_i=0$ otherwise, then we have
	\begin{align*}
	&Z_m(t_1,t_2,\dots,t_m)=Z_m(\underbrace{0,\dots 0}_{\ell-1},\sum_{z\in B}z^\ell a,\underbrace{0,\dots,0}_{\ell-1},\sum_{z\in B}z^{2\ell}a,\dots)=[u^m]\prod_{z\in B}(1-z^\ell u^\ell)^{-a/\ell}.
	\end{align*}
\end{lemma}
\begin{proof}
	Substituting the values of $t_i$ into \eqref{eq:exp}, we see that
	\begin{align*}
	Z_m(t_1,t_2,\dots,t_m)&=
	[u^m] \exp\Big(a\sum_{i=1}^\infty \frac{\sum_{z\in B}z^{\ell i} u^{\ell i}}{\ell i}\Big)\\
	&=[u^m] \exp\Big(-\frac{a}{\ell}\sum_{z\in B}\log (1-z^\ell u^\ell )\Big)\\
	&=[u^m]\prod_{z\in B}(1-z^\ell u^\ell)^{-a/\ell}.
	\end{align*}
\end{proof}

\begin{lemma}\label{lem:c}
Suppose that $p$ is an odd prime and $n$ is a positive integer. 	Let $G=\Z/2pn\Z$ be the additive group of order $2pn$, and let $D$ be the set defined  in \eqref{eq:D}. For a positive integer $m$, let $X=D^m$ and   
   define the character sum $S_m(\chi)$ by
	\begin{align}\label{eq:Smchi}
	S_m(\chi)=\frac{1}{m!}\sum_{(x_{1},x_{2},\dots,x_{m})\in \overline{X}}\chi(x_{1})\chi(x_{2})\cdots\chi(x_{m}),
	\end{align}
	for a character $\chi:G\to\myC$ on $G$.
	Denote by $o(\chi)$ the order of the character $\chi$. Then we have
	$$
	S_m(\chi)=\begin{cases}
	(-1)^m [u^m] (1-u^{o(\chi)})^{\frac{|D|}{o(\chi)}}, & \text{if}\ p\nmid\ o(\chi);\\
	(-1)^m[u^m] \prod_{k=1}^{(p-1)/2}(1-\overline{\chi}^{\frac{o(\chi)}{p}}(k) u^{\frac{o(\chi)}{p}})^{\frac{|G|}{o(\chi)}},& \text{if}\ p \mid \ o(\chi).
	\end{cases}$$
\end{lemma}
\begin{proof}
	Using  Theorem \ref{thm:Li-Wan}, we can write $S_m(\chi)$ as
	\begin{align}\label{eq:S1}
	S_m(\chi)=\frac{1}{m!}\sum_{\tau\in S_k}\sign(\tau)\sum_{(x_1,x_2,\dots,x_m)\in X_\tau}\chi(x_1)\chi(x_2)\cdots\chi(x_m).
	\end{align}
	Let $\tau=C_1\cdots C_{c(\tau)}$ be a disjoint cycle product of $\tau$. Then from the definition of $X_\tau$, we see that
	\begin{align}\label{eq:S2}
	\sum_{(x_1,x_2,\dots,x_m)\in X_\tau}\chi(x_1)\chi(x_2)\cdots\chi(x_m)=\prod_{i=1}^{c(\tau)} \Big(\sum_{x\in D} \chi^{\ell_i}(x)\Big),
	\end{align}
	where $\ell_i$ denotes the length of the cycle $C_i$, $1\leq i\leq c(\tau)$. %Thus we have to determine character sums over the set $D$.
	
	Let $[D]$ denote the image of $D$ under the quotient map $q:\Z\to G$ that sends $a$ to $a+2pn\Z$. We observe from \eqref{eq:D} that $[D]$ is a disjoint union of translations of the subgroup $pG$, where $pG=\{pg:g\in G\}$. Precisely, we have $[D]=\bigcup_{k=1}^{(p-1)/2}(pG-k)$, which implies that
	$$
	\sum_{x\in D}\chi(x)=\sum_{k=1}^{(p-1)/2}\sum_{x\in pG}\chi(x-k)=\Big(\sum_{k=1}^{(p-1)/2}\overline{\chi}(k)\Big)\sum_{x\in pG}\chi(x).
	$$
   The sum $\sum_{x\in pG}\chi(x)$ vanishes unless $\chi$ is a trivial character on $pG$ for which $o(\chi)=1$ or $p$. This implies that
   \begin{itemize}
		\item if $o(\chi)\neq 1,p$, then $\sum_{x\in D}\chi(x)=0$;
		\item if $o(\chi)=1$, then $\sum_{x\in D}\chi(x)=|D|$;
		\item if $o(\chi)=p$, then $\sum_{x\in D}\chi(x)=(\sum_{k=1}^{(p-1)/2}\overline{\chi}(k))|G|/p$.
	\end{itemize}
	It follows from the above classification that in the case $p\nmid o(\chi)$,  $\sum_{x\in D} \chi^{i}(x)=|D|$ if $o(\chi)\mid i$ and $\sum_{x\in D} \chi^{i}(x)=0$ otherwise; in the case $p\mid o(\chi)$,  $\sum_{x\in D}\chi^{i}(x)=(
	\sum_{k=1}^{(p-1)/2}\overline{\chi}^{i}(k))|G|/p$ if $\frac{o(\chi)}{p}\mid i$ and $\sum_{x\in D} \chi^{i}(x)=0$ otherwise. Thus we have two cases.

	{\bf Case 1: $p\nmid o(\chi)$.} In this case, we have $\sum_{x\in D}\chi^{i}(x)=|D|1_{o(\chi)\mid i}$. From \eqref{eq:S1} and \eqref{eq:S2}, we deduce that
	\begin{align*}
	S_m(\chi)&=\frac{1}{m!}\sum_{\tau\in S_m}\sign(\tau)\prod_{i=1}^\ell \Big(\sum_{x\in D} \chi^{\ell_i}(x)\Big)\\
	&=\frac{1}{m!}\sum_{\sum i c_i=m}N(c_1,c_2,\dots,c_m) (-1)^{m-\sum_{i=1}^m c_i}\prod_{i=1}^m (|D|1_{o(\chi)\mid i})^{c_i}\\
	&=(-1)^mZ_m(\underbrace{0,\dots,0}_{o(\chi)-1},-|D|,\underbrace{0,\dots,0}_{o(\chi)-1},-|D|,\dots)\\
	&=(-1)^m [u^m] (1-u^{o(\chi)})^{\frac{|D|}{o(\chi)}}.
	\end{align*}
	The last step is due to Lemma \ref{lem:a}.
	
	{\bf Case 2: $p\mid o(\chi)$}. In this case, we have $\sum_{x\in D}\chi^{i}(x)=(\sum_{k=1}^{(p-1)/2}\overline{\chi}^{i}(k))\frac{|G|}{p}1_{\frac{o(\chi)}{p}\mid i}$. A similar calculation as in case 1 shows that \begin{align*}
	S_m(\chi)&=(-1)^mZ_m(\underbrace{0,\dots,0}_{\frac{o(\chi)}{p}-1},\sum_{k=1}^{(p-1)/2}\overline{\chi}^{\frac{o(\chi)}{p}}(k)\frac{|G|}{p},\underbrace{0,\dots,0}_{\frac{o(\chi)}{p}-1},\sum_{k=1}^{(p-1)/2}\overline{\chi}^{\frac{2o(\chi)}{p}}(k)\frac{|G|}{p},\dots)\\
	&=(-1)^m[u^m] \prod_{k=1}^{(p-1)/2}(1-\overline{\chi}^{\frac{o(\chi)}{p}}(k) u^{\frac{o(\chi)}{p}})^{\frac{|G|}{o(\chi)}},
	\end{align*}
	where we used Lemma \ref{lem:b}.
	
\end{proof}

\begin{lemma}\label{lem:d}
	Let be $p$ be an odd prime, and let $r$ be an integer with $1\leq r\leq p-1$. Then we have
	$$\prod_{k=1}^{(p-1)/2}(1-e^{2\pi i k r/p})^2=pe^{\pi i (\frac{p^2-1}{4p}r+\frac{p-1}{2})}.$$
\end{lemma}
\begin{proof}
	Let $P$ denote the product on the left-hand side. Then we have
	\begin{align*}
	|P|^2=P\overline{P}&=\prod_{k=1}^{(p-1)/2}(1-e^{2\pi i k r/p})^2\prod_{k=1}^{(p-1)/2}(1-e^{-2\pi i k r/p})^2\\
	&=\prod_{k=1}^{(p-1)/2}(1-e^{2\pi i k r/p})^2\prod_{k=1}^{(p-1)/2}(1-e^{2\pi i (p-k) r/p})^2=\prod_{k=1}^{p-1}(1-e^{2\pi i k r/p})^2=p^2,
	\end{align*}
	where we used the fact that $\{e^{2\pi i k r/p},\ 1\leq k\leq p-1\}$ is a complete list of the primitive $p$-th roots of unity. This gives $|P|=p$. 
	
	 Then we are to determine the argument of $P$. From the elementary equality $(1-e^{i\theta})^2=2(1-\cos(\theta))e^{i(\theta+\pi)}$, we see that $\arg((1-e^{2\pi i rk/p})^2)=2\pi rk/p+\pi\mmod{2\pi}$, which yields
	\begin{align*}
	\arg(P)=  \sum_{k=1}^{(p-1)/2} (2\pi  rk/p+\pi)\mmod{2\pi}=\pi(\frac{p^2-1}{4p}r+\frac{p-1}{2})\mmod{2\pi}.
	\end{align*}
	
	Combining the results of modulus and argument of $P$, we conclude that
	$$P=\prod_{k=1}^{(p-1)/2}(1-e^{2\pi i k r/p})^2=pe^{\pi i (\frac{p^2-1}{4p}r+\frac{p-1}{2})}.
	$$
\end{proof}

%\begin{lemma}
%	Given a character $\chi\in\widehat{G}$, define $S_D(\chi):=\sum_{a\in D}\chi(a)$. Then we have
%	$$Z_k(-S_D(\chi),-S_D(\chi^2),\dots,-S_D(\chi^k))=[\frac{u^k}{k!}]\prod_{a\in D}(1-\chi(a)u).$$
%\end{lemma}
%\begin{proof}
%Let $\text{ord}(\chi)=d$. Then we have
%\begin{align*}
%Z_k(-S_D(\chi),-S_D(\chi^2),\dots,-S_D(\chi^k))&=[\frac{u^k}{k!}]\exp(\sum_{j=1}^d -S_D(\psi^j)\sum_{\ell=0}^\infty \frac{u^{d\ell+j}}{d\ell+j})\\
%& =[\frac{u^k}{k!}]\exp(\sum_{j=1}^d -S_D(\psi^j) \int_{0}^u \sum_{\ell=0}^\infty x^{d\ell+j-1}\,dx\\
%&=[\frac{u^k}{k!}]\exp(\sum_{j=1}^d -\sum_{a\in D}\chi^j(a) \int_{0}^u \frac{x^{j-1}}{1-x^d}\,dx\\
%&=[\frac{u^k}{k!}]\exp(\sum_{a\in D} \int_{0}^u \frac{\sum_{j=1}^d -\chi^j(a)x^{j-1}}{1-x^d}\,dx\\
%&=[\frac{u^k}{k!}]\exp(\sum_{a\in D}\int_0^u \frac{-\chi(a)}{(1-\chi(a)x)}\,dx\\
%&=[\frac{u^k}{k!}]\exp(\sum_{a\in D} \log (1-\chi(a)x)|_{0}^u)\\
%&=[\frac{u^k}{k!}]\exp(\sum_{a\in D}\log(1-\chi(a)u)\\
%&=[\frac{u^k}{k!}]\prod_{a\in D}({1-\chi(a)u})
%\end{align*}
%\end{proof}

\section{Proof of Theorem \ref{thm:main}}

As mentioned in the last part of Section \ref{sec:reduction}, the main problem is reduced to counting the quantity  $N_D(m_1,\dots,m_s,b)$. We shall use some character theory to estimate it. Again, let $G=\Z/2pn\Z$ be the additive group of order $2pn$ and let $D$ be the set defined in \eqref{eq:D}. Let $X_i=D^{m_i}$ be the  Cartesian product of $m_i$ copies of $D$, $1\leq i\leq s$. For an ordered $m$-tuple $x=(x_1,x_2,\dots,x_m)\in D^m$, denote by $\sigma(x):=\sum_{i=1}^m x_m$ the sum of its coordinates. Using the fact that $\frac{1}{|G|}\sum_{\chi\in G}\chi(x)$ is $1$ if $x=0_G$ and is $0$ otherwise,  we can express $N_D(m_1,m_2,\dots,m_2,b)$ as
\begin{align}\label{eq:NDD}
\begin{split}
&N_{D}(m_1,m_2,\dots,m_s,b)\\
 =&\frac{1}{m_1!m_2!\cdots m_s!}\sum_{(x_1,x_2,\dots,x_s)\in \overline{X_1}\times \overline{X_2}\times \cdots\times \overline{X}_s}1_{\sigma(x_1)+\sigma(x_2)+\cdots+\sigma(x_s)=b}\\
 =&\frac{1}{m_1!m_2!\cdots m_s!}\sum_{(x_1,x_2,\dots,x_s)\in \overline{X_1}\times \overline{X_2} \cdots\times \overline{X}_s}\frac{1}{|G|}\sum_{\chi\in\widehat{G}}\chi(\sigma(x_1)+\sigma(x_2)+\cdots+\sigma(x_s)-b)\\
=&\frac{1}{|G|}\sum_{\chi\in\widehat{G}}\overline{\chi}(b)\prod_{i=1}^s\frac{\sum_{x_i\in \overline{X_i}}\chi(\sigma(x_i))}{m_i!}\\
 =&\frac{1}{|G|} \sum_{\chi\in\widehat{G}}\overline{\chi}(b)\prod_{i=1}^s S_{m_i}(\chi),
\end{split}
\end{align}
where the character sum $S_{m_i}(\chi)$ is defined as in  \eqref{eq:Smchi}.

%Evaluating $S_m(\chi)$ a distinct coordinate counting problem that can be handled by the Li-Wan sieve.

Now we are ready to estimate $N_D(b)$. By Lemma \ref{lem:c} and equation \eqref{eq:NDD}, we have
\begin{align}
\begin{split}\label{eq:NDb0}
N_D(b)&=\sum_{0\leq m_i\leq |D|,1\leq i\leq s} (-1)^{\sum_{i=1}^s m_i}N_D(m_1,m_2,\dots,m_s,b)\\
     &\quad=\frac{1}{|G|}\Big(\sum_{\chi\in\widehat{G}:p\mid o(\chi)}\overline{\chi}(b)\sum_{0\leq m_i\leq |D|,1\leq i\leq s}\prod_{i=1}^s [u^{m_i}]\prod_{k=1}^{(p-1)/2}\big(1-\overline{\chi}^{\frac{o(\chi)}{p}}(k)u^{\frac{o(\chi)}{p}}\big)^{\frac{|G|}{o(\chi)}}\\
&\quad+\sum_{\chi\in\widehat{G}:p\nmid o(\chi)}\overline{\chi}(b)\sum_{0\leq m_i\leq |D|,1\leq i\leq s}\prod_{i=1}^s[u^{m_i}]\big(1-u^{o(\chi)}\big)^{\frac{|D|}{o(\chi)}}\Big).
\end{split}
\end{align}
We note that
\begin{align*}
\sum_{0\leq m_1,\dots,m_s\leq |D|}\prod_{i=1}^s [u^{m_i}]\prod_{k=1}^{(p-1)/2}\big(1-\overline{\chi}^{\frac{o(\chi)}{p}}(k)u^{\frac{o(\chi)}{p}}\big)^{\frac{|G|}{o(\chi)}}&=\prod_{i=1}^s\Big(\sum_{m_i=0}^{|D|}[u^{m_i}]\prod_{k=1}^{(p-1)/2}\big(1-\overline{\chi}^{\frac{o(\chi)}{p}}(k)u^{\frac{o(\chi)}{p}}\big)^{\frac{|G|}{o(\chi)}}\Big)\\
&=\prod_{k=1}^{(p-1)/2}(1-\overline{\chi}^{\frac{o(\chi)}{p}}(k))^{\frac{s|G|}{o(\chi)}}
\end{align*}
and similarly
\begin{align*}
\sum_{0\leq m_i\leq |D|,1\leq i\leq s}\prod_{i=1}^s[u^{m_i}](1-u^{o(\chi)})^{\frac{|D|}{o(\chi)}}=(1-1^{o(\chi)})^{\frac{s|D|}{o(\chi)}}=0.
\end{align*}
Substituting these into \eqref{eq:NDb0}, we obtain
\begin{align}\label{eq:NDb}
N_D(b)=\frac{1}{|G|}\sum_{\chi\in\widehat{G}:p\mid o(\chi)}\overline{\chi}(b)\prod_{k=1}^{(p-1)/2}(1-\overline{\chi}^{\frac{o(\chi)}{p}}(k))^{\frac{s|G|}{o(\chi)}}.
\end{align}

We split the sum \eqref{eq:NDb} into two parts:
\begin{align*}
N_D(b)&=\frac{1}{|G|}\sum_{\chi\in\widehat{G}:o(\chi)=p}\overline{\chi}(b)\prod_{k=1}^{(p-1)/2}(1-\overline{\chi}^{\frac{o(\chi)}{p}}(k))^{\frac{s|G|}{p}}+\frac{1}{|G|}\sum_{\chi\in\widehat{G}:p\mid o(\chi),o(\chi)>p}\overline{\chi}(b)\prod_{k=1}^{(p-1)/2}(1-\overline{\chi}^{\frac{o(\chi)}{p}}(k))^{\frac{s|G|}{o(\chi)}}\\
&=\frac{1}{|G|}\sum_{\chi\in\widehat{G}:o(\chi)=p}\overline{\chi}(b)\prod_{k=1}^{(p-1)/2}(1-\overline{\chi}(k))^{\frac{s|G|}{p}}+O\Big(\max_{\chi\in\widehat{G}:p\mid o(\chi),o(\chi)>p}\Big|\prod_{k=1}^{(p-1)/2}(1-\overline{\chi}^{\frac{o(\chi)}{p}}(k))\Big|^{\frac{s|G|}{2p}}\Big).
\end{align*}
Note that the implied constant in the big O notation can be $1$. Since $G=\Z/2pn\Z$,  we have $|G|=2pn$ and thus $N_D(b)$ is equal to
\begin{align}\label{eq:NDb2}
N_D(b)=\frac{1}{2pn}\sum_{\chi\in\widehat{G}:o(\chi)=p}\overline{\chi}(b)\prod_{k=1}^{(p-1)/2}(1-\overline{\chi}(k))^{2sn}+O\Big(\max_{\chi\in\widehat{G}:p\mid o(\chi),o(\chi)>p}\Big|\prod_{k=1}^{(p-1)/2}(1-\overline{\chi}^{\frac{o(\chi)}{p}}(k))\Big|^{sn}\Big).
\end{align}
Using Lemma \ref{lem:d}, we have
\begin{align}\label{eq:NDb3}
N_D(b)&=\frac{p^{sn}}{2pn}\sum_{r=1}^{p-1}e^{2\pi i \frac{b}{p}r}e^{ \pi i sn(\frac{p^2-1}{4p}r+\frac{p-1}{2})}+O(p^{sn/2}).
\end{align}
Note that for a character $\chi$ with $p\mid o(\chi)$, the character $\chi^{\frac{o(\chi)}{p}}$ is of order $p$. Thus Lemma \ref{lem:d} also implies that
$|\prod_{k=1}^{(p-1)/2}(1-\overline{\chi}^{\frac{o(\chi)}{p}}(k))|=\sqrt{p}$ for a character $\chi$ with $p\mid o(\chi)$ and consequently the error term in \eqref{eq:NDb2} is $O(p^{sn/2})$.

The sum in \eqref{eq:NDb3} is a geometric series with common ratio of $e^{2\pi i(\frac{b}{p}+sn\frac{p^2-1}{8p})}$, so $N_D(b)$ is equal to
\begin{align}\label{eq:NDb1}
N_D(b)=\frac{p^{sn}}{2pn}e^{\pi i sn\frac{p-1}{2}}(p-1)+O(p^{sn/2})=(-1)^{sn(p-1)/2}\frac{(p-1)p^{sn}}{2pn}+O(p^{sn/2}).
\end{align}
if $
b+sn(p^2-1)/8=0\mmod{p}$, and $N_D(b)$ is equal to
\begin{align}\label{eq:NDb4}
\begin{split}
N_D(b)%&=\frac{p^{sn}}{2pn}\frac{e^{\pi i (2\frac{b}{p}+sn(\frac{p^2-1}{4p}+\frac{p-1}{2}))}-e^{\pi i sn(\frac{p^2-1}{4}+\frac{p-1}{2})}}{1-e^{\pi i (\frac{2b}{p}+sn\frac{p^2-1}{4})}}+O(p^{sn/2})\\
&=\frac{p^{sn}}{2pn}e^{\pi i  sn\frac{p-1}{2}}\frac{e^{2\pi i (\frac{b}{p}+sn\frac{p^2-1}{8p})}-1}{1-e^{2\pi i (\frac{b}{p}+sn\frac{p^2-1}{8})}}+O(p^{sn/2})\\
&=-(-1)^{sn(p-1)/2}\frac{p^{sn}}{2pn}+O(p^{sn/2}),
\end{split}
\end{align}
otherwise, where we used $p^2-1=0\mmod {8}$ for an odd prime $p$.

%In summary, we have
%\begin{align}\label{eq:NDb1}
%N_D(b)=\begin{cases} (-1)^{sn(p-1)/2}\frac{(p-1)p^{sn}}{2pn}+O(p^{sn/2}), & \text{if}\ b+sn\frac{p^2-1}{8}=0\mmod{p};\\
%-(-1)^{sn(p-1)/2}\frac{p^{sn}}{2pn}+O(p^{sn/2}),& \text{othewise}.
%\end{cases}
%\end{align}

In view of \eqref{eq:Sb}, equations \eqref{eq:NDb1} and\eqref{eq:NDb4} imply that
$$S_{2pn,b}=\frac{(p-1)p^{sn}}{2pn}+O(p^{sn/2})$$ if
\begin{align}\label{eq:condition}
b-e+sn\cdot\frac{p^2-1}{8}=0\mmod{p},
\end{align}
and 
\begin{align*}
S_{2pn,b}=-\frac{p^{sn}}{2pn}+O(p^{sn/2})
\end{align*}
otherwise. Recall that $e=sn(p-1)(2pn+1-p)/8$. A direct calculation shows that the condition \eqref{eq:condition} can be simplified as
$$b-\frac{n(n-1)}{2}\cdot\frac{p-1}{2}\cdot sp=0\mmod{p},$$
which is equivalent to $p\mid b$, since $\frac{n(n-1)}{2}$ is always an integer and $p-1$ is an even number. 
This completes the proof. \qed

\end{document}